\documentclass{article}
\title{A modest improvement on the function $S(T)$}
\author{T. S. Trudgian\\
  Department of Mathematics and Computer Science,\\
  University of Lethbridge,\\
  Alberta, Canada, T1K 3M4\\
  \texttt{tim.trudgian@uleth.ca}}
\usepackage{amsthm}
\usepackage{amsmath}
\usepackage{amssymb}
\newtheorem{Lem}{Lemma}

\newtheorem{theorem}{Theorem}

\begin{document}
\maketitle

\begin{abstract}
This paper concerns the function $S(T)$, the argument of the Riemann zeta-function. Improving on the method of Backlund, and taking into account the refinements of Rosser and McCurley it is hereunder proved that for sufficiently large $T$
\begin{equation*}
|S(T)| \leq 0.1013 \log T.
\end{equation*}
Theorem \ref{ThemB} makes the above result explicit, viz.\ it enables one to select values of $a$ and $b$ such that, for $T>T_{0}$,
\begin{equation*}
|S(T)| \leq a + b\log T.
\end{equation*}
\end{abstract}

\section{Introduction}
Whenever $t$ does not coincide with an ordinate of a zero of $\zeta(\sigma +it)$ one defines the function $S(t)$ as
\begin{equation*}
S(t) = \pi^{-1}\arg\zeta(\tfrac{1}{2} +it),
\end{equation*}
where the argument is determined via continuous variation along the straight lines connecting $2, 2+it$ and $\frac{1}{2}+it$, with $S(0) = 0$. If $t$ is such that $\zeta(\sigma + it) = 0$ then define $S(t) = \frac{1}{2}\lim_{\epsilon \rightarrow 0}\{S(t-\epsilon) + S(t+\epsilon)\}.$ Without assuming unproven conjectures (for example the Riemann or Lindel\"{o}f hypotheses) the classic estimate of von-Mangoldt, $S(T) =O(\log T),$ has never been improved, except in reducing the size of the implied constant. Backlund \cite{Backlund1918} showed that, for $T\geq 200$,
\begin{equation*}
|S(T)| \leq 0.137 \log T + 0.445\log\log T + 4.35,
\end{equation*}
where the lower order terms were improved by Rosser \cite{Rossers}, who showed that, for $T\geq 1467$,
\begin{equation*}
|S(T)| \leq 0.137 \log T + 0.443\log\log T + 1.588,
\end{equation*}
and a computational check shows that this remains valid for all $T \geq 3$. Such explicit results are useful when estimating sums over the zeroes of $\zeta(s)$ --- see, e.g.\ \cite{Kadiri}, \cite{Sharp}.

The main idea of Backlund's Method is to count the number of zeroes of $\Re\zeta(\sigma +iT)$ on the line segment $[\frac{1}{2} + iT, 1+\eta + iT]$ where $0< \eta\leq \frac{1}{2}$. Suppose there are $n$ such zeroes, labelled $a_{1}, \ldots, a_{n}$. These zeroes partition the line segment $[\frac{1}{2} + iT, 1+\eta +iT]$ into $n+1$ intervals. On the interior of each interval, $\arg\zeta(s)$ can change by at most $\pi$, since by construction, $\Re\zeta(s)$ is non-zero on each interior. Thus, as $\sigma$ varies from $\frac{1}{2}$ to $1+\eta$ then
\begin{equation*}
|\Delta\arg\zeta(s)| \leq (n+1)\pi.
\end{equation*}
One proceeds to bound $n$ from above using Jensen's formula on the function
\begin{equation}\label{fdes}
f(s) = \frac{1}{2}\big\{\zeta(s+iT)^{N} + \zeta(s-iT)^{N}\big\},
\end{equation}
for $N$ a natural number\footnote{Backlund has $N=1$. The introduction of the number $N$ and the passing through a sequence of $N$s tending to infinity is due to Rosser. The advantages of this will be made plain on p.\ \pageref{Rtrick}.} --- thus $f(\sigma) = \Re \zeta(\sigma + iT)^{N}$. There are two ways to proceed. 

Method $\mathcal{A}$ takes account of all the zeroes contained in a circle of radius $r(\frac{1}{2}+\eta)$, centred at $s= 1+\eta + iT$, for some $r \in (1, 2]$. McCurley \cite{McCurley} follows this line of attack, with $r=2$. Contrarily, method $\mathcal{B}$ makes use of of a clever observation by Backlund, henceforth called `Backlund's trick'.

For any $\delta \in [0, \frac{1}{2} + \eta)$, let $\Delta_{1}\arg\zeta(s)$ denote the change in the argument of $\zeta(s)$ as $\sigma$ varies from $\frac{1}{2}$ to $\frac{1}{2} + \delta$. Similarly $\Delta_{2}\arg\zeta(s)$ is the change in argument as $\sigma$ varies from $\frac{1}{2}$ to $\frac{1}{2} - \delta$. By estimating the change in argument of $\chi(s)$, where
\begin{equation}\label{funct}
\zeta(s) = \chi(s)\zeta(1-s) = \pi^{s-\frac{1}{2}} \frac{\Gamma(\frac{1}{2} - \frac{1}{2}s)}{\Gamma(\frac{1}{2}s)} \zeta(1-s),
\end{equation}
(see, e.g., \cite[Ch.\ II]{Titchmarsh}) Backlund [pp.\ 355-357, \textit{op.\ cit.}] was able to show that for $T>1$
\begin{equation*}
|\Delta_{1}\arg\zeta(s) + \Delta_{2}\arg\zeta(s)| \leq \frac{8}{T},
\end{equation*}
It follows that there are at least\footnote{Alternatively, for large enough $T$ there are at least $n$ zeroes of $\Re\zeta(s)$. This matters precious little, especially in light of the improvements given by Rosser given on p.\ \pageref{Rtrick}.} $n-2$ zeroes of $\Re\zeta(s)$ on the line segment $[-\eta + iT, \frac{1}{2} +iT]$, and so at least $2n-2$ zeroes of $\Re\zeta(s)$ for $\sigma\in [-\eta, 1+\eta]$. So one uses Jensen's formula, with a circle of radius $1+2\eta$, centred at $s=1 + \eta +iT$. McCurley's argument\footnote{McCurley considers Dirichlet $L$-functions, whence he is unable to make use of Backlund's trick. Also, he considers $N(T)$ to be those zeroes with imaginary part$\gamma\in[-T, T].$ Thus the upper bound in (\ref{mcC}) is one quarter of that which is in \cite{McCurley}.} works here as well, and gives  [Thm 2.1, \textit{op.\ cit.}] 
\begin{equation}\label{mcC}
|S(T)| \leq 0.115\log T,
\end{equation}
for sufficiently large $T$.

The advantage of $\mathcal{B}$ is that one gets `2-for-the-price-of-1' in terms of the number of zeroes of $f(s)$. But the drawback is that one must estimate $|\zeta(s)|$ over the strip $-\eta \leq \sigma \leq 1+\eta$. With $\mathcal{A}$ one begins with fewer zeroes, but for a suitably small $r$, the incursion into the strip $\sigma \leq \frac{1}{2}$ is minimal. This is indeed an amelioration since, by convexity, $|\zeta(s)|$ grows much more quickly to the left of the line $\sigma = \frac{1}{2}$. Method $\mathcal{A}$ is that which is outlined in \cite[Ch. XIII, \S 9]{Titchmarsh}.

It needs must be noted that any detriment from using $\mathcal{B}$ is nullified if one uses the convexity bound $|\zeta(\frac{1}{2} +it)| \ll t^{1/4}$. Thus, if method $\mathcal{A}$ is of to be of any use to anyone, one must know the value of the constant $K$ for which $|\zeta(\tfrac{1}{2}+it)| \leq K t^{\theta}$ where $\theta < \frac{1}{4}$. Cheng and Graham \cite{Cheng} have shown that
\begin{equation}\label{cg}
|\zeta(\tfrac{1}{2} + it)| \leq 3 t^{\frac{1}{6}}\log t,
\end{equation}
for $t>e$, and this will be used in \S \ref{4.1}.

The remainder of the paper sets out to prove

\begin{theorem}[via Method $\mathcal{B}$]\label{thA}
If $t>t_{0}>e$ then
\begin{equation}\label{zxy}
|S(t)| \leq 1.998 + 0.17\log t.
\end{equation}
\end{theorem}

It should be noted that the theorem is valid for all $t>t_{0}>e$, and the particular choice of coefficients is that which minimises the right-side of (\ref{zxy}) when $t_{0} = 10^{10}$. Better bounds for larger values of $t_{0}$ are calculable from \S \ref{5}. The value of the coefficient of $\log t$ can be diminished further, but the limitations of the theorem show that it can not be taken to be less than $0.1027.$ Any diminution in this coefficient is at the expense of increasing the constant term. 

This paper can be considered a sequel to my paper on Turing's Method \cite{TrudTur}, and indeed many of the calculations involving convexity estimates for bounds on $|\zeta(\frac{1}{2} +it)|$ are similar.

\section{The requisites for Backlund's Method}\label{doi}
The opening gambits of Backlund and McCurley are essentially the same. One writes
\begin{equation*}
\xi(s) = \tfrac{1}{2}s(s-1)\pi^{-\frac{1}{2}s}\Gamma(\tfrac{1}{2}s)\zeta(s), 
\end{equation*}
where $\xi(s)$ is an entire function, whose zeroes coincide with the non-trivial zeroes of $\zeta(s)$. If one writes $N(T)$ as the number of complex zeroes of $\zeta(s)$ with imaginary part $\gamma\in [0,T]$, then it follows from Cauchy's theorem, the functional equation and the reflection principle that $4\pi N(T) = 4\Delta_{R}\arg\xi(s)$, where $R$ is the pair of lines connecting the points $1+\eta$, $1+\eta + iT$ and $\frac{1}{2}+ iT$. In calculating the change in argument of $\xi(s)$ one finds a main term and then the term corresponding to $\Delta_{R}\arg\zeta(s)$ which is\footnote{One can use Cauchy's theorem and the fact that $\arg\zeta(2) = \arg\zeta(1+\eta) = 0$ to show that calculating $\Delta\arg\zeta(s)$ along the aforementioned lines agrees with the definition of $S(T)$.} $\pi S(T)$. The vertical piece is easily handled, since here, $|\arg\zeta(s)| \leq |\log \zeta(s)| \leq \log\zeta(1+\eta)$. What remains is to estimate $\Delta_{h}\arg\zeta(s)$: the change in argument of $\zeta(s)$ along the line segment $[1+\eta + iT, \frac{1}{2} +iT]$.

With $f(s)$ defined as in (\ref{fdes}), it follows that
\begin{equation*}
|\Delta_{h} \arg\zeta(s)| = \frac{1}{N}|\Delta_{h}\arg\zeta(s)|^{N} \leq \frac{(n+1)\pi}{N},
\end{equation*}
whence 
\begin{equation}\label{gvib}
|S(T)| \leq \frac{2}{\pi}\log\zeta(1+\eta) + \frac{n+1}{N}.
\end{equation}
One can now produce an upper bound on $n$ courtesy of method $\mathcal{A}$ or $\mathcal{B}$. The proof below is valid for any $r\in(1, 2]$ and it difference between the two methods will be plainly seen.

\section{Bounding $n$ using method $\mathcal{A}$}\label{MetA}

For $r\in(1,2]$, Jensen's formula is applied to the function $f(s)$ on a circle with radius $r(\frac{1}{2} + \eta)$ centred at $s = 1+\eta + iT$, to give
\begin{equation}\label{biggun}
\begin{split}
n\log r &\leq \frac{1}{2\pi} \int_{-\pi/2}^{3\pi/2} \log |f(1+\eta + r(\tfrac{1}{2} + \eta)e^{i\phi})|\, d\phi - \log |f(1+\eta)|\\
&= I_{1} + I_{2} + I_{3} + I_{4} -\log |f(1+\eta)|,
\end{split}
\end{equation}
Where $I_{1}$ covers $\phi\in [-\frac{\pi}{2}, \frac{\pi}{2}]$; $I_{2}$ covers $\phi\in [\frac{\pi}{2}, \frac{\pi}{2} + \sin^{-1}r^{-1}]$; $I_{3}$ covers $\phi\in [\frac{\pi}{2} + \sin^{-1}r^{-1},  \frac{3\pi}{2} - \sin^{-1}r^{-1}]$ and $I_{4}$ covers $\phi\in [\frac{3\pi}{2} - \sin^{-1}r^{-1}, \frac{3\pi}{2}].$

After this herculean dividing of ranges of integration one notes that, $\Re (s)\geq 1+\eta$ on $I_{1}$, whence one estimates $f(s)$ trivially, viz.
\begin{equation*}\label{I_{1}}
\log |f(s)| \leq N \log \zeta(1+\eta).
\end{equation*}
On both $I_{2}$ and $I_{4}$, $\Re(s) \geq \frac{1}{2}$, so that it is only on $I_{3}$ that $-\eta \leq \Re (s) \leq \frac{1}{2}$ --- and this contribution diminishes as $r$ is taken closer and closer to unity. 

To handle the $\log|f(1+\eta)|$ term, one makes use of the trick of Rosser. Write $\zeta(1+\eta + iT) = ke^{i\psi}$. Now choose a sequence of $N$s tending to infinity such that $N\psi$ tends to $0$ modulo $2\pi$, whence 
\begin{equation*}\label{Rtrick}
\lim_{N\rightarrow\infty} \frac{f(1+\eta)}{|\zeta(1+\eta + iT)|^{N}} = 1.
\end{equation*}
Finally, for $\sigma >1$, one can consider the Euler product of $\zeta(s)$ to show that $|\zeta(s)| \geq \frac{\zeta(2\sigma)}{\zeta(\sigma)}$, whence the bound 
\begin{equation*}\label{f0}
-\log |f(1+\eta)| \leq N\log \zeta(1+\eta).
\end{equation*}
The only terms in (\ref{biggun}) left to estimate are $I_{2}$ and $I_{3}$; a bound for $I_{2}$ will serve as a bound for $I_{4}$. Estimates of the growth of $|\zeta(s)|$ for in $\sigma \in [-\eta, \frac{1}{2}]$ and for in $\sigma\in [\frac{1}{2}, 1+\eta]$ are given in the following section.

\section{Preliminary Results}
An explicit version of the Phragm\'{e}n--Lindel\"{o}f theorem is needed and this is given below in
\begin{Lem}\label{lem1}
Let $a, b, Q$ and $k$ be real numbers, and let $ f(s)$ be regular analytic in the strip $-Q\leq a\leq \sigma \leq b$ and satisfy the growth condition 
$$ |f(s)| <C\exp \left\{e^{k|t|}\right\},$$ for a certain $C>0$ and for $0<k<\pi/(b-a)$. Also assume that
\[|f(s)|\leq\left\{\begin{array}{ll}
A|Q+s|^{\alpha} &  \mbox{for $\Re(s) = a,$}\\
B|Q+s|^{\beta} & \mbox{for $\Re(s) = b$}\\
\end{array}
\right.\]
with $\alpha \geq \beta$. Then throughout the strip $a \leq \sigma \leq b$ the following holds
$$ |f(s)| \leq A^{(b-\sigma)/(b-a)}B^{(\sigma - a)/(b-a)}|Q+s|^{\alpha(b-\sigma)/(b-a) + \beta(\sigma-a)/(b-a)}.$$
\end{Lem}
\begin{proof}
This is a result of Rademacher and can be found in \cite[pp.\ 66-67]{Rad}.
\end{proof}

In order to apply Lemma \ref{lem1} one needs bounds on $|\zeta(s)|$ on each of the three lines: $\sigma = 1+\eta$, \, $\sigma = \frac{1}{2}$, \, and $\sigma = -\eta$. Trivially,
\begin{equation}\label{1+e}
|\zeta(1+\eta + iT)| \leq \zeta(1+\eta).
\end{equation}
The bound of Cheng and Graham (\ref{cg}) may be used on the line $\sigma = \frac{1}{2}$. One can bound $|\zeta(-\eta + iT)|$ by using the functional equation (\ref{funct}), (\ref{1+e}) and the following result due to Rademacher. 

\begin{Lem}\label{lem2}
For $-\frac{1}{2} \leq\sigma\leq \frac{1}{2}$,
\begin{equation*}
\Bigg|\frac{\Gamma(\frac{1}{2} - \frac{1}{2}s)}{\Gamma(\frac{1}{2}s)}\Bigg| \leq \left(\frac{|1+s|}{2}\right)^{\frac{1}{2}-\sigma}.
\end{equation*}
\end{Lem}
\begin{proof}
See \cite[p.\ 197]{Rademacher}.
\end{proof}

It follows that 
\begin{equation}\label{-e}
|\zeta(-\eta + iT)| \leq \left(\frac{|s+1|}{2\pi}\right)^{\frac{1}{2} + \eta} \zeta(1+\eta).
\end{equation}

The following lemma contains two estimates on the growth of $|\zeta(s)|$ in strips on either side of the critical line. 

\begin{Lem}\label{subconeslem}
Suppose there exist constants $B$ and $\theta$ satisfying
\begin{equation}\label{opsi}
|\zeta(\tfrac{1}{2} + it)| \leq B|s+1|^{\theta},
\end{equation} 
for all $t$. Equations (\ref{1+e}) and (\ref{opsi}) show that for $\frac{1}{2} \leq \sigma \leq 1+\eta$, and for $t>t_{0}>e$
\begin{equation}\label{rightofahalf}
|\zeta(s)| \leq \left\{C_{1}^{\theta(1+\eta - \sigma) + \frac{1}{2} + \eta}B^{1+\eta - \sigma}\log\zeta(1+\eta)^{\sigma - \frac{1}{2}} t^{\theta(1+\eta - \sigma)}\right\}^{1/(\frac{1}{2} + \eta)},
\end{equation}
where
\begin{equation*}
C_{1} = \sqrt{1+ \left(\frac{2+\eta}{t_{0}}\right)^{2}}.
\end{equation*}

Also, given (\ref{opsi}) and (\ref{-e}) then for $-\eta \leq \sigma \leq \frac{1}{2}$ and for $t>t_{0} >e$,
\begin{equation}\label{leftofahalf}
|\zeta(s)| \leq \left\{\left[\frac{\zeta(1+\eta)}{(2\pi)^{\frac{1}{2}+  \eta}}\right]^{\frac{1}{2} - \sigma}B^{\sigma + \eta} \{C_{2}t\}^{(\frac{1}{2} + \eta)(\frac{1}{2} - \sigma) + \theta(\sigma+\eta)}\right\}^{1/(\frac{1}{2} + \eta)},
\end{equation}
where 
\begin{equation*}
C_{2} = \sqrt{1+\frac{1}{t_{0}^{2}}}.
\end{equation*}
\end{Lem}

The writing of (\ref{opsi}) is simply to use a suitable form of (\ref{cg}) in Lemma \ref{lem1}. To prove (\ref{rightofahalf}) take\footnote{Note that Lemma \ref{lem1} cannot be applied directly to $\zeta(s)$ owing to the pole at $s=1$.}, in Lemma \ref{lem1}, $f(s) = (s-1)\zeta(s), \, a = \frac{1}{2}, \, b= 1+\eta,\, Q=1$, and use (\ref{opsi}) and (\ref{1+e}). The term $C_{1}$ springs from replacing $|s-1|$ with $|s+1|$.

To prove (\ref{leftofahalf}) take $f(s) = \zeta(s),\, a= -\eta,\, b = \frac{1}{2},\, Q=1$, and use (\ref{-e}) and (\ref{opsi}). The term $C_{2}$ is obtained by replacing $|s+1|$ with $t$. 

\subsection{The value of $B$}\label{4.1}
To arrive at (\ref{opsi}) consider (\ref{cg}), viz.
\begin{equation*}
|\zeta(\tfrac{1}{2} +it)| \leq 3t^{1/6}\log t \leq 3|s+1|^{1/6} \log t,
\end{equation*}
for $t>e$. To accommodate the $\log t$ term, note that one can choose a small $\delta$ and hence find a (large) $A_{0} = A_{0}(\delta, t_{0})$ such that $\log t \leq A_{0} t^{\delta} \leq A_{0}|s+1|^{\delta}$, for $t\geq t_{0}$. Since the function $\log t/t^{\delta}$ never exceeds $(\delta e)^{-1}$, it follows that for all $t>e$
\begin{equation*}
|\zeta(\tfrac{1}{2} + it)| \leq  \frac{3}{\delta e}|s+1|^{1/6 + \delta},
\end{equation*}
and a computational check shows the above to be valid for all $t\geq 0$. Thus we may take 
\begin{equation}\label{bdef}
B= B(\delta) =  \frac{3}{\delta e}.
\end{equation}
However, this presupposes that at a reasonable height for computation one wishes to use the bound (\ref{cg}) as opposed to the `ordinary' convextiy estimate
\begin{equation}\label{nopsi}
|\zeta\left(\tfrac{1}{2} +it\right)| \leq 2.53\,  |1+s|^{\frac{1}{4}},
\end{equation}
which can be deduced from that in \cite{Lehman}.
It is clear that as $t$ increases one should prefer (\ref{cg}) to (\ref{nopsi}) but, as will be shown in the next section, this preference is not immutable, particularly for modest values of $t$. Indeed, the dependence of $B$ on $\delta$ is the primary source of frustration in seeking an improvement to Backlund's method, and it would be of great use to have access to a bound of the type
\begin{equation*}
|\zeta(\tfrac{1}{2} + it)|\leq Ct^{1/6},
\end{equation*}
which would be of use even if $C$ were as large as, say $1000$.

\section{Computation}\label{5}

Equation (\ref{gvib}) is
\begin{equation*}
|S(T)| \leq \frac{2}{\pi}\log\zeta(1+\eta) + \frac{n+1}{N},
\end{equation*}
where $n$ is bounded by (\ref{biggun}). One can now use Lemma \ref{subconeslem} in (\ref{biggun}) to obtain a bound on $S(T)$ depending on, \textit{inter alia}, the variable $r$ where $1< r\leq 2$. This general form is bloated with terms involving $\sin^{-1}1/r$ and the like, and to include it here would be inexcusable. One must decide whether to use Backlund's trick (i.e. $r=2$ and twice as many zeroes) or to take a smaller value of $r$. 

It can be shown, after a little computation, that the use of Backlund's trick is the better option. The general bound of (\ref{gvib}) is given in an appendix, and hereafter we shall choose $r=2$. For ease of exposition many of the error terms have been estimated --- probably not optimally\footnote{For example, using an upper bound $\eta \leq 1$, while true, is a weaker estimate for many of the applications. But since many of these terms are suitably small, and since Theorem \ref{ThemB} concisely presents the nature of the upper bound for $S(T)$, such minute savings have been ignored.} --- and the $r=2$ upper bound on (\ref{gvib}) is given below in 

\begin{theorem}\label{ThemB}
When $\theta = \frac{1}{6} + \delta$ choose $B(\delta) = \frac{3}{\delta e}$ so that (\ref{opsi}) holds; if $\theta = \frac{1}{4}$ choose $B(\delta) = 2.53$, whence (\ref{opsi}) holds by virtue of (\ref{nopsi}). In either case, for all $T> T_{0}>3$
\begin{equation}\label{BackTur}
|S(T)|\leq a+ b\log T,
\end{equation}
where
\begin{equation}\label{aet}
a = a(\delta, \eta, T_{0}) =  1.85\log\zeta(1+\eta) + 0.71\log B(\delta) - 0.58 + \frac{1}{T_{0}},
\end{equation}
and 
\begin{equation}\label{bet}
b =  b(\delta, \eta) =  \frac{2\theta(1+\tfrac{\pi}{3} - \sqrt{3}) + (\eta + \tfrac{1}{2})( \sqrt{3} - \tfrac{\pi}{3})}{2\pi\log 2}.
\end{equation}
\end{theorem}

Equation (\ref{bet}) shows that the size of $b$ is diminshed when $\eta$ and $\delta$ are taken closer to zero. Also, (\ref{aet}) shows that $a$ is increased when either of $\eta$ or $\delta$ is diminished. This inverse proportionality occurs similarly in analysis of Turing's Method \cite{TrudTur} and it is herewith treated in like fashion.

If one wishes to investigate the size of $S(T)$ beyond some large height, then one can afford to take $\delta$ and $\eta$ smaller, so long as the term $b\log T$ in (\ref{BackTur}) continues to dominate. Indeed, for a given $T_{0}$, the minimal value of $a+ b\log T_{0}$ is sought. As an example, the Riemann hypothesis has been verified past $T_{0} = 10^{10}$ (see, e.g.\ \cite{Wed}) so it is beyond this height that explicit bounds on $S(T)$ would be of the greatest use. 

As a benchmark, Rosser's bounds on $|S(T)|$ are, for $T\geq T_{0}$
\begin{equation}\label{rbound}
|S(T)| \leq 1.588 + \left\{0.137 + 0.443\frac{\log\log T_{0}}{\log T_{0}}\right\}\log T.
\end{equation}

The following table compares the size of $b$ --- the coefficient of $\log T$ in (\ref{BackTur}) --- and the overall bound on $S(T)$, where each is obtained by Rosser's bound (\ref{rbound}), Theorem \ref{ThemB} with $\theta = \frac{1}{4}$, and Theorem \ref{ThemB} with $\theta = \frac{1}{6} + \delta$.
\begin{table}[h]
\caption{Comparison of bounds on $S(t)$}
\label{optable}
\centering
\begin{tabular}{c c c c c c c}
\hline\hline
$t_{0}$ & \multicolumn{2}{c}{(\ref{rbound})} &\multicolumn{2}{c}{Thm \ref{ThemB}: $\theta = \frac{1}{4}$} &\multicolumn{2}{c}{Thm \ref{ThemB}: $\theta = \frac{1}{6} + \delta$} \\[2 ex] 
& $b$ & $S(T)$ & $b$ & $S(T)$ & $b$ &  $S(T)$\\[0.5 ex] \hline
$10^{10}$ & 0.1974 & 6.132 & 0.170 & 5.912& 0.170 & 7.968 \\ 
$10^{12}$ & 0.1902 & 6.844 & 0.162 & 6.67& 0.162 & 8.644 \\ 
$10^{14}$ & 0.1847 & 7.543 & 0.156 & 7.395& 0.156 & 9.298\\ 
$10^{16}$ & 0.1804 & 8.233 & 0.152 & 8.122& 0.152 & 9.932\\ 
$10^{18}$ & 0.1768 & 8.916 & 0.148 & 8.797& 0.148 & 10.56\\ 
$10^{20}$ & 0.1738 & 9.594 & 0.145 & 9.47& 0.145 & 11.17\\ 
$10^{40}$ & 0.159 & 16.21 & 0.131 & 15.78& 0.126 & 17.26\\ 
$10^{60}$ & 0.153 & 22.70 & 0.126 & 21.69& 0.119 & 22.44\\ \hline\hline
\end{tabular}
\end{table}

Theorem \ref{thA} follows at once from the first row of middle-column, along with the calculation of $a$ from (\ref{aet}). Note that the convexity estimates are marginally superior to Rosser's bounds in each case. Moreover, the sub-convexity estimates (the right-column) only improve on Rosser's bounds in the last row. A simple computation shows that the value of $b$ obtained from the sub-convexity estimates, only overtakes that obtained by the middle-column when $T_{0} > 10^{26}$. 

Finally, note that from \cite{Huxley}, the bound $\zeta(\frac{1}{2} +it) \ll t^{\theta}$, where $\theta = \frac{32}{205}$ and Theorem \ref{ThemB} show that
\begin{equation*}
|S(T)| \leq 0.1013\log T,
\end{equation*}
for $T$ sufficiently large.

\section{Conclusion}

It is tempting to see what further improvements to Theorem \ref{thA} might be possible. One way is to try to combine methods $\mathcal{A}$ and $\mathcal{B}$. That is, to take some $r<2$ and to try to replicate Backlund's trick by showing that there must be some zeroes of $f(s)$ lying on the segment left of $\frac{1}{2} +iT$, that is, the line connecting $1+\eta - r(\frac{1}{2} + \eta) + iT$ and $\frac{1}{2} +iT$. Unfortunately such a manoeuvre would require some knowledge of the nature of the horizontal distribution of the zeroes of $\Re \zeta(s)$. If such a result were known it would be natural to suspect some diminution in the constants in Theorem \ref{thA}.

\section{Appendix: the explicit bound of method $\mathcal{A}$}
For any $r\in(1, 2)$ then
\begin{equation}\label{bloated}
|S(T)| \leq \frac{2}{\pi}\log\zeta(1+\eta) + \frac{a_{1}+ a_{2}\log B + a_{3}\frac{9}{2t_{0}^{2}} + a_{4} \log\zeta(1+\eta) + a_{5}\log t}{\pi\log r},
\end{equation}
where
\begin{equation*}
\begin{split}
a_{1}& = \tfrac{3\pi}{8t_{0}} -(\tfrac{1}{2}\log 2\pi)\tfrac{\pi}{2} - \sin^{-1}\tfrac{1}{r} + r\sqrt{1-\tfrac{1}{r^{2}}}, \\
a_{2} & = 2\left(\tfrac{\pi}{2} - \sin^{-1}\tfrac{1}{r}\right) - 3r\sqrt{1-\tfrac{1}{r^{2}}} +r, \\
a_{3} &= r\theta(2-\sqrt{1-\tfrac{1}{r^{2}}}) + \sin^{-1}\tfrac{1}{r} + (\tfrac{1}{2} + \eta - 2\theta)(\sin^{-1}\tfrac{1}{r} -\tfrac{\pi}{2} + r\sqrt{1-\tfrac{1}{r^{2}}}), \\
a_{4} &= \tfrac{3\pi}{2} - \sin^{-1}\tfrac{1}{r} + 2r\sqrt{1-\tfrac{1}{r^{2}}} + 1-r, \\
a_{5} &= r\theta +(\tfrac{1}{2} + \eta - 2\theta)(\sin^{-1}\tfrac{1}{r} -\tfrac{\pi}{2} + r\sqrt{1-\tfrac{1}{r^{2}}}).
\end{split}
\end{equation*}
The bound in (\ref{bloated}) only improves on that in Theorem \ref{thA} if $\zeta|(\frac{1}{2} +it)| \ll t^{\theta}$, where $\theta < 1/50$.

\section*{Acknowledgements}
I wish to thank Nathan Ng and Habiba Kadiri who suggested this problem, and Roger Heath-Brown for his helpful suggestions.

\bibliographystyle{plain}
\bibliography{themastercanada}

\end{document}